\documentclass[12pt]{amsart}
\usepackage{amssymb, amscd, amsmath, dev}

\theoremstyle{plain}
\newtheorem{thm}{Theorem}[section]

\newtheorem{cor}[thm]{Corollary}

\newtheorem{lem}[thm]{Lemma}

\newtheorem{rem}[thm]{Remark}

\def\A{{\mathcal A}}                              
\def\C{{\mathbb C}}                               
\def\F{{\mathbb F}}                               
\def\P{{\Phi}}                                    %
\def\R{{\mathbb R}}                               
\def\W{{\mathcal W}}                              
\def\Z{{\mathbb Z}}                               

\def\a{{\alpha}}                                  
\def\b{{\beta}}                                   
\def\bp{{\overline{\partial}}}                    
\def\bz{{\overline{z}}}                           
\def\l{{\lambda}}                                 %
\def\m{{\mu}}                                     %
\def\n{{\mathcal N}}                              %
\def\p{{\partial}}                                
\def\r{{\rho}}                                    %
\def\s{{\sigma}}                                  %
\def\x{{\zeta}}                                   %
\def\z{{\mathcal Z}}                              %

\def\GL{\operatorname{GL}}                        
\def\Hom{\operatorname{Hom}}                      

\def\det{\operatorname{det}}                      
\def\iso{{\stackrel{\sim}{~\longrightarrow~}}}    
\def\op{{\oplus}}                                 
\numberwithin{equation}{section}

\begin{document}
\title[Gelfand models for finite Coxeter groups]{On Gelfand models for finite Coxeter groups}
\author[Shripad M. Garge, Joseph Oesterl{\'e}]{Shripad M. Garge, Joseph Oesterl{\'e}}
\address{Shripad M. Garge. Department of Mathematics, Indian Institute of Technology Bombay, Powai, Mumbai. 400 076. INDIA.}
\email{shripad@math.iitb.ac.in}
\address{Joseph Oesterl\'e. Institut de Math{\'e}matiques de Jussieu, 16 rue Clisson, 75013 Paris. FRANCE.}
\email{oesterle@math.jussieu.fr}

\begin{abstract}
A Gelfand model for a finite group $G$ is a complex linear representation of $G$ that contains each of its irreducible representations with multiplicity one. 
For a finite group $G$ with a faithful representation $V$, one constructs a representation which we call the polynomial model for $G$ associated to $V$. 
Araujo and others have proved that the polynomial models for certain irreducible Weyl groups associated to their canonical representations are Gelfand models. 

In this paper, we give an easier and uniform treatment for the study of the polynomial model for a general finite Coxeter group associated to its canonical representation. 
Our final result is that such a polynomial model for a finite Coxeter group $G$ is a Gelfand model if and only if $G$ has no direct factor of the type $W(D_{2n}), W(E_7)$ or $W(E_8)$.
\end{abstract}

\maketitle


\begin{centerline}
{\small {\dn \399wFpAdQyA aAI k\4\314w sO\314w mAltF m\314w gg\?{\qvb} yA\2QyA -\9{m}(yT\0}}
\end{centerline}

\section{Introduction}
Let $G$ be a finite group. 
A {\em Gelfand model} for $G$ is a complex linear representation of $G$ that contains each of its irreducible representations with multiplicity one. 
One is interested in finding ``natural'' Gelfand models for classes of finite groups.

Klyachko and others (\cite{K, IRS, Ba}) gave a construction of Gelfand models for the groups $W(A_n), W(B_n)$ and $W(D_{2n + 1})$. 
It is obtained by taking sum of inductions of certain one dimensional representations of involution centralizers, hence called {\em involution model}. 
It is easy to see that, for a finite group $G$, the dimension of an involution model for $G$ is equal to the dimension of a Gelfand model for $G$ (i.e., the sum of dimensions of all irreducible representations of $G$) if and only if every irreducible complex linear representation of $G$ can be realized over $\R$. 
Therefore, if a group $G$ has a non-real irreducible representation then any involution model for $G$ is never a Gelfand model. 
Thus, the study of involution models is rather restricted. 

In this paper, we study another approach for constructing a Gelfand model. 
For a finite group $G$ and a faithful representation $V$ of $G$, we define the {\em polynomial model for $G$ associated to $V$}, denoted by $\n(V)$, to be the space of complex valued polynomial functions on $V$ that are annihilated by all $G$-invariant differential operators with polynomial coefficients of negative degree (see $\S 2$ for more details). 
Araujo and others (\cite{AA, A, AB}) proved that the polynomial models associated to the canonical representations of the Weyl groups $W(A_n), W(B_n)$ and $W(D_{2n + 1})$ are Gelfand models. 

The purpose of this paper is to study polynomial models for all finite Coxeter groups, irreducible or not, associated to their canonical faithful representations. 
In Theorem \ref{thm:2}, we give another description of the polynomial model $\n(V)$ associated to a faithful representation $V$ of a finite group $G$.
Our description is easier to work with and gives much more information than the original description in \cite{AA}. 
For instance, it is quite clear from Theorem \ref{thm:2} that an irreducible representation of $G$ is always contained in any polynomial model. 
Thus, $\n(V)$ is a Gelfand model for $G$ if and only if the multiplicity of each irreducible representation of $G$ in $\n(V)$ is one. 

The question now is to compute the multiplicity of an irreducible representation of $G$ in its polynomial model $\n(V)$. 
If $G$ is a finite Coxeter group, then this question is related to the study of the fake degree of an irreducible representation of $G$. 
In Section 3, we recall some well-known facts about finite Coxeter groups and the fake degrees. 
Section 4 is devoted to proving the following main result.

\subsection*{Theorem \ref{thm:main}}
{\em Let $G$ be a finite Coxeter group and let $V$ be its canonical faithful representation. 
The polynomial model $\n(V)$ for $G$ is a Gelfand model if and only if $G$ has no direct factor of the type $W(D_{2n}), W(E_7)$ or $W(E_8)$.}

\subsection*{Remark} 
Every irreducible representation of a finite group $G$ is contained in any polynomial model for $G$ (see Theorem \ref{thm:2}). 
Hence a polynomial model for $G$ is a Gelfand model if and only if its dimension is equal to the dimension of a Gelfand model. 
On the other hand, as discussed above, the dimension of an involution model for a finite Coxeter group $G$ is always equal to the dimension of a Gelfand model for $G$. 
Therefore, an involution model is a Gelfand model for a finite Coxeter group $G$ if and only if it contains each of its irreducible representations. 

Thus, in the case of finite Coxeter groups, involution models and polynomials models have complementary properties and it would be interesting to study the interplay between these two models. 
We hope to take up this study in a subsequent paper.

We also believe that our description of the polynomial model in Theorem \ref{thm:2} can be used to describe Gelfand models for finite groups which are not Coxeter groups. 
A trivial example is that the polynomial model for a finite cyclic group $G$ associated to any of its one dimensional faithful representations is a Gelfand model for $G$. 

\subsection*{Acknowledgements}
The first author thanks D.-N. Verma for introducing him to Gelfand models. 
Thanks are also due to Dipendra Prasad, Ulf Rehmann and Anuradha Garge for their comments at various stages of this work. 
This work was done when the first author was visiting the Institut de Math{\'e}matiques de Jussieu, Paris, and Universit\"{a}t Bielefeld, Germany.
He acknowledges the hospitality of the members of both the places as well as a fellowship from ``R{\'e}gion Ile-de-France'' and the support from ``SFB 701: Spektrale Strukturen und Topologische Methoden in der Mathematik'' which made these visits possible. 
The authors also thank the anonymous referee for a careful and speedy report. 

\section{The polynomial model}
In this section, we define the polynomial model for a finite group $G$ together with a faithful $\C$-representation $V$. 
This is a certain subspace of the ring of complex valued polynomial functions on $V$. 
In Theorem \ref{thm:2}, we give another description of this representation space which turns out to be easier and more useful than the one given in \cite{AA}. 

Let $G$ be a finite group and let $V$ be a faithful $\C$-linear representation of $G$. 
Let $\A$ denote the ring of complex valued polynomial functions on $V$. 
Then $G$ acts on $\A$ by 
$$(g \cdot P)(v) = P(g^{-1} \cdot v), {\rm ~for~all} ~g \in G, P \in \A {\rm   ~and~} v \in V.$$
Let $\W$ denote the Weyl algebra consisting of differential operators on $\A$ with polynomial coefficients. 
Then $\W$ admits an action of $G$ determined by the following property:
$$(g \cdot D)(P) = g \cdot (D (g^{-1} \cdot P)), {\rm~for~all} ~g \in G, D \in \W {\rm ~and~} P \in \A .$$
By choosing a basis of $V$, we identify $\A$ with $\C[x_1, \dots, x_n]$ and the Weyl algebra $\W$ with $\C [x_1, \dots, x_n, \p_1, \dots, \p_n]$, where $\p_i = \frac{\p}{\p x_i}$. 
Note that the multiplication in $\W$ is non-commutative. 
Any element $D \in \W$ can be written in a unique way as a finite sum $D = \sum c_{\a, \b} x^{\a} \p^{\b}$ where $x^{\a} = x_1^{\a_1} \cdots x_n^{\a_n}, \p^{\b} = \p_1^{\b_1} \cdots \p_n^{\b_n}$ and $c_{\a, \b} \in \C$. 
For $D = \sum c_{\a, \b} x^{\a} \p^{\b} \in \W$, we define the degree of $D$ as
$$\deg(D) = \max \left\{ \sum (\a_i - \b_i): c_{\a, \b} \ne 0\right\} .$$
Note that the degree is invariant under $G$-action. 
Let $\z$ be the set of $G$-invariant elements of negative degree in $\W$. 
Finally, let $\n(V) \subseteq \A$ denote the space of polynomials annihilated by $\z$, i.e., 
$$\n(V) := \{P \in \A: D(P) = 0 {\rm ~for ~all~} D \in \z\} .$$
It is clearly $G$-stable. 
We define the $G$-representation $\n(V)$ to be the {\em polynomial model for $G$ associated to $V$}. 

We now begin our study of the polynomial model $\n(V)$ with the following basic observation.

\begin{lem}
There exists a $G$-equivariant injective map $:\C[G] \hookrightarrow \A$. 
\end{lem}

\begin{proof}
Observe that, since $V$ is a faithful representation of $G$, there exists a vector $v \in V$ such that $g \cdot v \ne v$ for $1 \ne g \in G$. 
Further, there exists a polynomial $P \in \A$ such that $P(v) = 0$ and $P(g \cdot v) \ne 0$ for $g \ne 1$. 
We define the required map by sending $g$ to $g \cdot P$. 
\end{proof}

Let $\A_m$ denote the subspace of homogeneous polynomials in $\A$ of degree $m$. 
Since $G$ acts linearly on $\A$, each $\A_m$ is stable under the action of $G$.
It now follows from the above lemma that each irreducible representation of $G$ is contained in some $\A_m$. 

Further, let $\W_{p, q}$ denote the subspace of $\W$ generated by the elements $x^{\a} \p^{\b}$ where $|\a| := \sum_i \a_i = p$ and $|\b| := \sum_i \b_i = q$. 
Again, each $\W_{p, q}$ is stable under $G$-action. 
We recall that $\Hom_{\C} (\A_q, \A_p)$ admits an action of $G$ determined by $(g \cdot \phi) (P) = g \cdot (\phi (g^{-1} \cdot P))$ where $\phi \in \Hom_{\C} (\A_q, \A_p)$ and $P \in \A_q$. 
By letting $\W$ act on $\A$, we get a $G$-equivariant linear map $\Psi: \W \rightarrow \Hom_{\C} (\A, \A)$. 
From this map, we obtain a $G$-equivariant linear map $\Psi_{p, q}: \W_{p, q} \rightarrow \Hom_{\C} (\A_q, \A_p)$ for each $p$ and $q$. 

\begin{lem}
The $G$-equivariant map $\Psi_{p, q}: \W_{p, q} \rightarrow \Hom_{\C}(\A_q, \A_p)$ is an isomorphism. 
\end{lem}

\begin{proof}
Let $(u_{\a, \b})$, for $|\a| = p$ and $|\b| = q$, be the basis of $\Hom_{\C}(\A_q, \A_p)$ defined by $u_{\a, \b}(x^{\gamma}) = 0$ for $\gamma \ne \b$ and $u_{\a, \b}(x^{\b}) = x^{\a}$. 
It can be easily seen that $\Psi_{p, q}$ maps $x^{\a} \p^{\b}$ to $(\prod_i \b_i!) u_{\a, \b}$. 
It then follows that the map $\Psi_{p, q}$ is an isomorphism. 
\end{proof}

\begin{cor}\label{cor}
The map $\Psi_{p, q}$ induces an isomorphism $\W_{p, q}^G \iso \Hom_{\C[G]}(\A_q, \A_p)$.
\end{cor}

\begin{proof}
$\big(\Hom_{\C}(\A_q, \A_p)\big)^G = \Hom_{\C[G]}(\A_q, \A_p)$. 
\end{proof}

We now state and prove the main theorem of this section.

\begin{thm}\label{thm:2}
Let $V$ be a faithful representation of a finite group $G$. 
Let $\n(V)$ denote the polynomial model for $G$ associated to $V$.  
For an irreducible representation $U$ of $G$, let $p(U)$ be the smallest integer such that $U$ is isomorphic to a subspace of $\A_{p(U)}$ and let $C_U$ denote the $U$-isotypical component of $\A_{p(U)}$. 
Then 
$$\n(V) = \op C_U .$$ 
\end{thm}

\begin{proof}
Any $D \in \z$ maps $\A_{p(U)}$ to $\op_{p < p(U)} \A_p$, hence maps $C_U$ to $0$. 
Thus, $C_U$ is contained in $\n(V)$. 
Further, if some $\A_q$, for $q > p(U)$, has a subspace $U'$ isomorphic to $U$, then there exists a non-zero $G$-equivariant map from $\A_q$ to $\A_{p(U)}$. 
This, by Corollary \ref{cor}, comes from an element $D \in \W_{p(U), q}^G$ such that $D(U') \ne 0$. 
Since $q > p(U)$, we have $\deg(D) < 0$ and hence $D \in \z$. 
This proves that the only homogeneous component of $\n(V)$ which has a subspace isomorphic to $U$ is of degree $p(U)$. 
\end{proof}

\begin{cor}\label{cor:poly:Gelfand}
With the above notations, the polynomial model $\n(V)$ is a Gelfand model for $G$ if and only if each irreducible representation $U$ of $G$ appears with multiplicity one in its first occurrence in the homogeneous components of $\A$. 
\end{cor}

We complete this section with an easy observation and its two consequences. 

\begin{lem}\label{lem:product}
Let $G_1, G_2$ be finite groups with faithful representations $V_i$ for $i = 1, 2$. 
Let $G = G_1 \times G_2$ and let $W = V_1 \oplus V_2$. 
Then the polynomial model $\n(W)$ for $G$ associated to the faithful representation $W$ is isomorphic to $\n(V_1) \otimes \n(V_2)$ as a $G$-module. 
\end{lem}

\begin{proof}
Let $\A(W), \A(V_1)$ and $\A(V_2)$ denote the rings of complex valued polynomial functions on the vector spaces $W, V_1$ and $V_2$ respectively.
It is then clear that $\A(W) = \A(V_1) \otimes \A(V_2)$ and $\A(W)_m = \sum_{i + j = m} \A(V_1)_i \otimes \A(V_2)_j$. 
Any irreducible representation of $G$ is of the form $U = U_1 \otimes U_2$, where $U_i$ are irreducible representations of $G_i$, for $i = 1, 2$.
We then have $p(U) = p(U_1) + p(U_2)$ and $C_U = C_{U_1} \otimes C_{U_2}$. 
The lemma now follows from Theorem \ref{thm:2}. 
\end{proof}

\begin{cor}
With notations as in the above lemma, $\n(W)$ is a Gelfand model for $G$ if and only if $\n(V_i)$ is a Gelfand model for $G_i$ for $i = 1, 2$. 
\end{cor}

\begin{cor}\label{cor:product}
If $G_i$, for $1 \leq i \leq n$, are finite groups with corresponding faithful representations $V_i$. 
Let $G = \Pi_i G_i$ and $W = \oplus_i V_i$. 
Then the polynomial model $\n(W)$ is a Gelfand model for $G$ if and only if $\n(V_i)$ is a Gelfand model for $G_i$ for every $1 \leq i \leq n$. 
\end{cor}

\section{Finite Coxeter groups and fake degrees}
In this section, we recall some basic facts about finite Coxeter groups and their representations. 
More precisely, we recall the notion of the fake degree of an irreducible representation of a finite Coxeter group. 
The basic results about Coxeter groups are recalled from \cite{B} and those about the fake degree are recalled from \cite{C}. 
Towards the end of the section, we describe the fake degrees of irreducible representations of some finite Coxeter groups. 

A {\em Coxeter system} is a pair $(G, S)$ where $G$ is a group generated by a finite set $S$ with relations $s^2 = (ss')^{m(s, s')} = 1$ for $s, s' \in S$. 
If $(G, S)$ is a Coxeter system, by abuse of notation, we call $G$ a {\em Coxeter group}. 
In this paper, we are concerned with {\em finite Coxeter groups}. 
Let $G$ be a finite Coxeter group and let $V$ denote the $\C$-vector space generated by the basis $(e_s)_{s \in S}$.
Then $G$ admits a natural action on $V$ determined by the following property:
$$s (e_{s'}) = e_{s'} + 2 \cos \left(\frac{\pi}{m(s, s')}\right) e_s .$$
This is a canonical faithful representation of $G$ associated to the generating set $S$. 
In the rest of this paper, we shall not mention the generating set $S$ and say that $\P:G \hookrightarrow \GL(V)$ is the canonical faithful representation of the Coxeter group $G$. 

As described in the previous section, $G$ acts on the ring $\A$ of complex valued polynomial functions on $V$. 
Let $\A^G$ denote the subalgebra of $G$-invariant polynomials in $\A$. 
It is known that $\A^G$ is a polynomial algebra over $\C$ and, moreover, one has that $\A^G = \C[f_1, \dots, f_n]$ where the polynomials $f_i$ can be chosen to be homogeneous. 
The generating set $\{f_1, \dots, f_n\}$ is not unique, however, the degrees $d_i = \deg f_i$ are uniquely determined by the group $G$ and its representation $V$ (\cite[2.4.1]{C}). 

The quotient ring of $\A$ modulo the ideal generated by non-constant $G$-invariants, $\bar{\A} = \A/(f_i)$, admits a natural $G$-action. 
We decompose $\bar{\A}$ into ($G$-stable) homogeneous components as $\bar{\A} = \op_{i \geq 0} \bar{\A}_i$. 
It is known that $\bar{\A}$ is isomorphic to the regular representation of $G$ as a $G$-representation (\cite[2.4.6]{C}). 
For an irreducible representation $U$ of $G$ of degree $d$, let $a_1 \leq \dots \leq a_d$ denote the degrees of homogeneous components of $\bar{\A}$, counted with multiplicities, that contain a subrepresentation isomorphic to $U$. 

The polynomial $f_U(t) := \sum_{i = 1}^d t^{a_i}$ is called the {\em fake degree} of the representation $U$. 

If $\chi$ denotes the character of the representation $U$, then the fake degree of $U$ can be computed using the formula (\cite[11.1.1]{C})
\begin{eqnarray}\label{eqn:fake}
f_U(t) & = & |G|^{-1} \prod_{i = 1}^n (1 - t^{d_i}) \sum_{g \in G} \frac{\chi(g)}{\det(1 - \P(g) t)} ,
\end{eqnarray}
where $d_i$ are the degrees of the homogeneous generators $f_i$ of $\A^G$, as discussed above. 

\begin{rem}\label{rem:mult}
{\em If we write $f_U(t) = t^{q(U)} \cdot g_U(t)$, where $g_U(t)$ is a polynomial with non-zero constant term, then $q(U)$ is equal to $p(U)$, the smallest degree of the homogeneous component of $\A$ which has a subrepresentation isomorphic to $U$. 
Further, the constant term of $g_U(t)$ is equal to the multiplicity of the representation $U$ in $\A_{p(U)}$.}
\end{rem}

The finite Coxeter groups are classified by the corresponding Coxeter graphs (\cite[IV.1.9]{B}). 
A finite Coxeter group is said to be {\em irreducible} if its Coxeter graph is connected, and a general finite Coxeter group is a (finite) direct product of irreducible ones.
The irreducible finite Coxeter groups have been classified and, up to isomorphism, they are of the following types:
\begin{enumerate}
\item Classical Weyl groups: $W(A_n), W(B_{n + 1})$ and $W(D_{n + 3})$ for $n \geq 1$;
\item Exceptional Weyl groups: $W(G_2), W(F_4), W(E_6), W(E_7)$ and $W(E_8)$;
\item Non-crystallograhic groups: $H_3, H_4$ and the dihedral groups $G_2(n)$ of order $2n$ for $n = 5$ and $n \geq 7$. 
\end{enumerate}
We shall use these notations for the rest of this paper. 

The fake degrees of irreducible representations have been computed for various finite Coxeter groups. 
Steinberg calculated them for the groups $W(A_n)$ in \cite{St} whereas Lusztig and others computed them for many of the remaining finite Coxeter groups (\cite{L, BL, AL}). 
We now reproduce the formulae for the fake degrees of the irreducible representations of classical irreducible Weyl groups from (\cite[$\S$2]{L}). 
In the end, we also compute the fake degrees for the groups $H_3$ and $G_2(n)$. 

\subsection*{Fake degrees for $W(A_n)$} It is known that irreducible representations of $W(A_n) \cong S_{n + 1}$ are in 1-1 correspondence with partitions of $n + 1$, $\a:\a_1 \geq \a_2 \geq \cdots \geq \a_m \geq 0$ with $|\a| = \sum \a_i = n + 1$. 
For a partition $\a$, we define $\l_i = \a_i - i + m$. 
Then the fake degree of the corresponding representation $U_{\a}$ is given by 
\begin{equation}\label{eqn:A_n}
f_{U_{\a}}(t) = \frac{\prod_{i = 1}^n (t^i - 1)} {t^{\binom{m - 1}{2} + \binom{m - 2}{2} + \cdots + 1}} \cdot \frac{\prod_{\l_i > \l_j} (t^{\l_i} - t^{\l_j})}{\prod_{\l_j} \prod_{i = 1}^{\l_j} (t^i - 1)} .
\end{equation}

\subsection*{Fake degrees for $W(B_n)$} The irreducible representations of $W(B_n)$ are in 1-1 correspondence with ordered pairs $(U_{\a}, U_{\b})$ of irreducible representations of $S_k$ and $S_l$ where $k + l = n$ (see \cite[$\S$2.3]{L} for the details of this correspondence). 
For the partitions $\a:\a_1 \geq \cdots \geq \a_{m'} \geq 0$ and $\b:\b_1 \geq \cdots \geq \b_{m''} \geq 0$, we define $\l_i = \a_i - i + m'$ $(1 \leq i \leq m')$ and $\m_i = \b_i - i + m''$ $(1 \leq i \leq m'')$. 
Then the fake degree of the representation $U_{\a, \b}$, corresponding to the ordered pair $(U_{\a}, U_{\b})$, is given by 
\begin{equation}\label{eqn:B_n}
f_{U_{\a, \b}}(t) = \frac{t^{|\b|} \prod_{i = 1}^n (t^{2i} - 1)} {t^{2 \binom{m' - 1}{2} + \cdots + 2} \cdot t^{2 \binom{m'' - 1}{2} + \cdots + 2}} \cdot \frac{\prod_{\l_i > \l_j} (t^{2 \l_i} - t^{2 \l_j})}{\prod_{\l_j} \prod_{i = 1}^{\l_j} (t^{2i} - 1)} \cdot \frac{\prod_{\m_i > \m_j} (t^{2 \m_i} - t^{2 \m_j})}{\prod_{\m_j} \prod_{i = 1}^{\m_j} (t^{2i} - 1)} .
\end{equation}

\subsection*{Fake degrees for $W(D_n)$} The Weyl group $W(D_n)$ is a subgroup of index $2$ in $W(B_n)$. 
It is known that the restriction $\widetilde{U}_{\a, \b} := U_{\a, \b}|_{W(D_n)}$, for an irreducible representation $U_{\a, \b}$ of $W(B_n)$, remains irreducible if $\a \ne \b$ and that $\widetilde{U}_{\a, \b} \cong \widetilde{U}_{\b, \a}$. 
Further, $\widetilde{U}_{\a, \a}$ splits into two distinct irreducible representations $\widetilde{U}_{\a, \a}'$ and $\widetilde{U}_{\a, \a}''$. 
All irreducible representations of $W(D_n)$ are obtained in this way (\cite[11.4.4]{C}). 
The fake degrees of the irreducible representations of $W(D_n)$ are as follows
\begin{equation}\label{eqn:D_n:1}
f_{\widetilde{U}_{\a, \a}'}(t) = f_{\widetilde{U}_{\a, \a}''}(t) = \frac{(t^n - 1)f_{U_{\a, \a}}(t)}{t^{2n} - 1}
\end{equation}
and for $\a \ne \b$
\begin{equation}\label{eqn:D_n:2}
f_{\widetilde{U}_{\a, \b}}(t) = \frac{(t^n - 1)(f_{U_{\a, \b}}(t) + f_{U_{\a, \b}}(t))}{t^{2n} - 1} .
\end{equation}

\begin{lem}\label{lem:classical}
For an irreducible representation $U$ of a finite Coxeter group $G$, we decompose the fake degree of $U$ as $f_U(t) = t^{p(U)} \cdot g_U(t)$ such that the constant term of $g_U(t)$ is non-zero $($as in Remark $\ref{rem:mult}$$)$. 

If $G$ is $W(A_n), W(B_n)$ or $W(D_{2n + 1})$ then the constant term of the polynomial $g_U(t)$ is equal to $1$ for all irreducible representations $U$ of $G$. 
Further, there exist irreducible representations $U$ of $W(D_{2n})$ such that the constant term of $g_U(t)$ is $2$. 
\end{lem}

\begin{proof}
We use the notations: $U_{\a}, U_{\a, \b}$ and $\widetilde{U}_{\a, \b}$, as above, for the irreducible representations of the groups $W(A_n), W(B_n)$ and $W(D_n)$ respectively. 

It is clear from the formulas (\ref{eqn:A_n}) and (\ref{eqn:B_n}) for the fake degrees $f_{U_{\a}}(t)$ and $f_{U_{\a, \b}}(t)$, given above, that for any irreducible representation $U$ of $W(A_n)$ or $W(B_n)$, the constant term of $g_U(t)$ is $1$. 

Further, for the group $W(B_n)$, we observe from (\ref{eqn:B_n}) that $p(U_{\a, \b}) = p(U_{\b, \a})$ implies that $|\a| = |\b|$ and hence $n = 2 |\alpha|$. 
Therefore, the irreducible representations $U_{\a, \b}$ and $U_{\b, \a}$ of $W(B_{2n + 1})$ have the property that $p(U_{\a, \b}) \ne p(U_{\b, \a})$. 
As a result, we get that the constant term of the polynomial $g_U(t)$ for an irreducible representation $U$ of $W(D_{2n + 1})$ is always $1$. 

Finally, we note that the constant term of the polynomial $g_{\widetilde{U}_{\a, \b}}(t)$, with $|\a| = |\b|$ and $\a \ne \b$, is $2$. 
\end{proof}

\subsection*{Fake degrees for $H_3$.} The fake degrees of the irreducible representations of the group $H_3$ can be easily computed using the formula \ref{eqn:fake}. 
The group $H_3$ is isomorphic to the direct product of the alternating group $A_5$ with $\Z/2\Z = \{1, -1\}$. 
The group $A_5$ has five irreducible representations, which we denote by $U_1 (=$trivial), $V_4$, $W_5$, $Y_3$ and $Z_3$, where the subscripts indicate the degrees of the representations. 
Each irreducible representation of $A_5$ gives rise to $2$ irreducible representations of $H_3$ depending on whether the action of $\Z/2\Z$ is trivial or not. 
Thus, we get $10$ irreducible representations of $H_3$. 
We denote the irreducible representations in the same way as for $A_5$ if the action of $\Z/2\Z$ is trivial and the ones with the nontrivial action of $\Z/2\Z$ are denoted by $U_1', V_4', W_5', Y_3'$ and $Z_3'$. 
It is clear that any of the $Y_3'$ and $Z_3'$ can be taken as a defining representations of $H_3$. 
We choose $Y_3'$. 
The degrees of $H_3$ are $2, 6$ and $10$. 
Then the fake degrees of the irreducible representations of $H_3$ are as follows:
\begin{center}
\begin{tabular}{clcl}
$U_1$: & $1$, & $U_1'$: & $t^{15}$, \\
$V_4$: & $t^4 + t^6 + t^8 + t^{12}$, & $V_4'$: & $t^3 + t^7 + t^9 + t^{11}$, \\
$W_5$: & $t^2 + t^4 + t^6 + t^8 + t^{10}$, \hskip3mm & $W_5'$: & $t^5 + t^7 + t^9 + t^{11} + t^{13}$, \\
$Y_3$: & $t^6 + t^{10} + t^{14}$, & $Y_3'$: & $t^1 + t^5 + t^9$, \\
$Z_3$: & $t^8 + t^{10} + t^{12}$, & $Z_3'$: & $t^3 + t^5 + t^7$. \\ 
\end{tabular}
\end{center}

\subsection*{Fake degrees for $G_2(n)$} 
The degree of an irreducible representation of a dihedral group is at most two. 
The group $G_2(n)$ is generated by a rotation $\rho$ of order $n$ and a reflection $\sigma$ with relation $\rho \sigma = \sigma \rho^{-1}$. 
The group $G_2(n)$ has $[\frac{n - 1}{2}]$ representations of degree $2$. 
These representations $W_j$, for $1 \leq j \leq [\frac{n - 1}{2}]$, can be described by the images of $\sigma$ and $\rho$ as follows: 
$$\sigma \mapsto \begin{pmatrix} 1 & 0 \\ 0 & -1 \end{pmatrix} \hskip2mm {\rm and} \hskip2mm \rho \mapsto \begin{pmatrix} \cos j \theta & \sin j \theta \\ - \sin j \theta & \cos j \theta \end{pmatrix} ,$$
where $\theta = \frac{2 \pi \sqrt{-1}}{n}$. 
We take $W_1$ to be the defining representation of the Coxeter group $G_2(n)$. 
By using formula (\ref{eqn:fake}) and the fact that the degrees of $G_2(n)$ are $2$ and $n$, one has $f_{W_j}(t) = t^j + t^{n - j}$. 

If $n = 2k + 1$, then $G_2(n)$ has $2$ one dimensional representations with fake degrees $1$ and $t^{2k + 1}$. 
If $n = 4k$, then it has $4$ representations of degree one and their fake degrees are $1, t^k, t^k$ and $t^{2k}$. 

\section{Polynomial models for finite Coxeter groups}
In this section, we prove the main result of this paper (Theorem \ref{thm:main}). 

By our description of the polynomial model in Theorem \ref{thm:2}, it is enough to study the polynomial models for {\em irreducible} finite Coxeter groups. 
Again using the same description, we connect the study of polynomial models for a finite Coxeter group $G$ with the study of fake degrees of its irreducible representations. 
We then prove the main result using some classical results of Steinberg, Lusztig and others which are recalled at the end of previous section. 
Finally, we give an explicit construction of the polynomial model for dihedral group $G_2(n)$ corresponding to its canonical faithful representation $W_1$. 

\begin{thm}\label{thm:irr}
Let $G$ be an irreducible finite Coxeter group and let $V$ be its canonical faithful representation.
Then the corresponding polynomial model $\n(V)$ of $G$ associated to $V$, constructed as in $\S 2$, is a Gelfand model for $G$ if and only if $G$ is not of the type $W(D_{2n}), W(E_7)$ or $W(E_8)$. 
\end{thm}

\begin{proof}
Let $G$ be a finite Coxeter group and let $V$ denote the canonical faithful representation of $G$. 
We recall, from Corollary \ref{cor:poly:Gelfand}, that the polynomial model $\n(V)$ of $G$ is a Gelfand model if and only if the multiplicity of each irreducible representation $U$ of $G$ in $\A_{p(U)}$ is one. 
By Remark \ref{rem:mult}, we further deduce that $\n(V)$ is a Gelfand model for $G$ if and only if for each irreducible representation $U$ of $G$, the polynomial $g_U(t)$ (associated to the fake degree $f_U(t)$ of $U$) has constant term one. 

The irreducible finite Coxeter groups constitute Weyl groups of split simple linear algebraic groups, dihedral groups, the symmetry group of the icosahedron, denoted by $H_3$, and a group which is denoted by $H_4$ (Bourbaki \cite[6.4.1]{B}). 

Beynon-Lusztig, in \cite[$\S \S$ 2, 4, 5]{BL}, have computed the fake degrees of irreducible representations for all exceptional Weyl groups. 
These computations, combined with Lemma \ref{lem:classical}, imply that when $G$ is a Weyl group of a simple algebraic group, the polynomial model $\n(V)$ for $G$ is a Gelfand model if and only if $G$ is not of the type $W(D_{2n}), W(E_7)$ or $W(E_8)$. 

Similarly, from the computation of the fake degrees of irreducible representations of the group $H_4$ in \cite[$\S$ 3]{AL}, it follows that the polynomial model $\n(V)$ associated to the canonical faithful representation of $H_4$ is a Gelfand model. 

Finally, from our computation carried out in the previous section, it follows that the polynomial model for dihedral groups and the group $H_3$ is a Gelfand model. 
\end{proof}

\begin{thm}\label{thm:main}
Let $G$ be a finite Coxeter group and let $V$ be its canonical faithful representation. 
The polynomial model $\n(V)$ for $G$ is a Gelfand model for $G$ if and only if $G$ has no direct factor of the type $W(D_{2n}), W(E_7)$ or $W(E_8)$. 
\end{thm}

\begin{proof}
Follows from Theorem \ref{thm:irr} and Corollary \ref{cor:product}. 
\end{proof}

\subsection*{The polynomial model for $G_2(n)$}
In the case of the dihedral group $G_2(n)$ one can explicitly describe the polynomial model corresponding to the representation $W_1$.
Under this representation the images of the elements $\rho$ and $\sigma$ are as follows:
$$\sigma \mapsto \begin{pmatrix} 1 & 0 \\ 0 & -1 \end{pmatrix} \hskip2mm {\rm ~and~} \hskip2mm \rho \mapsto \begin{pmatrix} \cos 2 \pi/n & \sin 2 \pi/n \\ - \sin 2 \pi/n & \cos 2 \pi/n \end{pmatrix} .$$
The group $G_2(n)$ then acts on the polynomial ring $\A = \C[x_1, x_2]$ which we treat as the ring of polynomials in $z$ and $\bz$ over $\C$, where $z = x_1 + i x_2$ and $\bz = x_1 - i x_2$. 
Similarly, the Weyl algebra $\W$ can be treated as the algebra $\C [z, \bz, \p, \bp]$ where $\p = \frac{\p}{\p z}$ and $\bp = \frac{\p}{\p \bz}$. 
The action of $G_2(n)$ on $z, \bz, \p, \bp$ is described by:
\begin{equation}\label{eqn:dihedral:1}
\s(z) = \bz, ~\s(\bz) = z, ~\s(\p) = \bp, ~\s(\bp) = \p,
\end{equation}
\begin{equation}\label{eqn:dihedral:2}
~\r(z) = \x z, ~\r(\bz) = \x^{-1} \bz, ~\r(\p) = \x^{-1} \p ~{\rm and}~ \r(\bp) = \x \bp
\end{equation}
where $\x = e^{2 \pi i/n}$. 

\begin{lem}
With the above notations, the polynomial model $\n(W_1)$, for the group $G_2(n)$ corresponding to its representation $W_1$, is generated as a vector space by the basis:
$$\{1, z, \dots, z^{[\frac{n}{2}]}, \bz, \dots, \bz^{[\frac{n}{2}]}, z^n - \bz^n\} ,$$
where $z = x_1 + i x_2$ and $\bz = x_1 - i x_2$.
\end{lem}

\begin{proof}
Observe that, by formulae (\ref{eqn:dihedral:1}) and (\ref{eqn:dihedral:2}), it follows that $\p \bp$ and $z^{n - m}\bp^m + \bz^{n - m} \p^m$, where $m = [\frac{n}{2}] + 1$, belong to $\z$. 
It follows that $\n(W_1)$ is contained in the vector subspace of $\A$, say $V$, generated by $\{1, z, \dots, z^{[\frac{n}{2}]}, \bz, \dots, \bz^{[\frac{n}{2}]}, z^n - \bz^n \}$. 
Further, by Theorem \ref{thm:main}, $\n(W_1)$ contains each irreducible representation of $G_2(n)$ and the sum of degrees of irreducible representations of $G_2(n)$ is $2 [\frac{n}{2}] + 2$ (\cite{S}). 
It therefore follows that $\n(W_1) = V$ and that $\n(W_1)$ is indeed a Gelfand model for $G$. 
\end{proof}

\end{document}